\documentclass[11pt]{article}

\usepackage{amsmath,amsfonts,amssymb,amsthm,bbm}
\usepackage{graphics,epsfig}
\usepackage{color}
\usepackage{array}

\newtheorem{theorem}{Theorem}[section]

\newtheorem{corollary}{Corollary}[section]

\newtheorem{remark}{Remark}[section]

\numberwithin{equation}{section}

\allowdisplaybreaks

\def\qed{\hfill$\Box$\medskip}

\begin{document}

\noindent\makebox[60mm][l]{\tt {\large}}

\bigskip
%%%%%%%%%%%%%%%%%%%%%%%%%%%%%%% Title %%%%%%%%%%%%%%%%%%%%%%%%%%%%%%%%%%%%%%\\
\noindent{
\begin{center}
\Large\bf A note on the passage time of  finite state  Markov chains\footnote{{  The project is partially supported by NSFC Grant No.11131003.}}
\end{center}
}

\noindent{
\begin{center}Wenming Hong\footnote{School of Mathematical Sciences
\& Laboratory of Mathematics and Complex Systems, Beijing Normal
University, Beijing 100875, P.R. China. Email: wmhong@bnu.edu.cn}  \ \ \
 Ke Zhou\footnote{ {  School of Statistics, University of International Business and Economics, Beijing 100029, P.R. China. Email: zhouke@uibe.edu.cn}}
\end{center}
}

\vspace{0.1 true cm}

%%%%%%%%%%%%%%%%%%%%%%%%%%%%%% Abstrct %%%%%%%%%%%%%%%%%%%%%%%%%%%%%%%%%%%%
\begin{center}
\begin{minipage}[c]{12cm}
\begin{center}\textbf{Abstract}\end{center}
\bigskip
Consider a Markov chain  with finite state $\{0, 1, \cdots, d\}$.  We give the generation functions (or Laplace transforms) of absorbing  (passage) time in the following two situations : (1) the absorbing time of state $d$ when the chain starts from any state $i$ and absorbing at state $d$; (2) the passage time of any state $i$ when the chain starts from the stationary distribution {  supposed} the chain is time reversible and ergodic.  Example shows that it is more convenient compared with the existing methods, especially we can calculate the expectation of the absorbing time directly.

\

%   A well-known theorem  for an irreducible skip-free chain with absorbing state $d$, under some conditions,
 % is that the hitting time starting from state $0$ to state $d$ is distributed as the sum of $d$ independent
% geometric (or exponential) random variables. The purpose of this paper is to present a direct and simple proof
 % of the theorem in the cases of both discrete  and continuous time  skip-free Markov chains. Our proof is to
 % calculate directly the generation functions (or Laplace transforms) of hitting times in terms of the iteration method.\\

\mbox{}\textbf{Keywords:}\quad Markov chain, absorbing time, passage time,
generation functions,
Laplace transforms, eigenvalues, stationary distribution. \\
\mbox{}\textbf{Mathematics Subject Classification (2010)}:  60E10,
60J10, 60J27, 60J35.
\end{minipage}
\end{center}

%%%%%%%%%%%%%%%%%%%%%%%%%%%%% Section 1 %%%%%%%%%%%%%%%%%%%%%%%%%%%%%%%%%%%%\\

%%%%%%%%%%%%%%%%%%%%%%%%%%%%% Section 2 %%%%%%%%%%%%%%%%%%%%%%%%%%%%%%%%%%%%\\
\section{ Discrete time\label{s2}}
\subsection{Absorbing  time when the process starting from any   state $i$ }
Consider  the discrete time Markov chain $\{X_n\}_{n\geq 0}$ with finite states $\{0, 1, \cdots, d\}$ and absorbing
at state $d$,  the transition probability matrix $P$ is given by
%$P$ is the transition probability matrix.
\begin{equation*}
%\label{p}
P=\left(
 \begin{array}{cccccc}
  r_0 & p_{0,1}  &  & \cdots & p_{0,d-1} & p_{0,d} \\
   q_{1,0} & r_1  & & \cdots & p_{1,d-1}& p_{1,d}\\
   %q_{2,0} & q_{2,1} & r_3 & p_2 & \\
   %q_{30} & q_{31} & q_{32} & r_4 & p_3 & \\
   \vdots & \ddots  & & \ddots & \ddots & \vdots\\
   q_{d-1,0}& q_{d-1,1}& & \cdots &  r_{d-1} & p_{d-1,d} \\
    0 & 0 & & \cdots & 0 & 1\\
     \end{array}
      \right)_{(d+1)\times(d+1)}.
\end{equation*}
%\begin{equation*}
% P_n=\left(
%             \begin{array}{cccccc}
%               r_0 & p_0 & \\
%               q_{1,0} & r_1 & p_1 & \\
%               %q_{2,2} & q_{2,1} & r_2 & p_2 & \\
%               %q_{31} & q_{32} & q_{33} & r_4 & p_3 & \\
%               \ddots & \ddots & \ddots & \ddots & \ddots & p_{n-1}\\
%               q_{n,0}& q_{n,1}& q_{n,2}& \cdots & \cdots &  r_{n} \\
%             \end{array}
%           \right)_{(n+1)\times(n+1)},~~\text{for}~ 0\leq n\leq d-1.
%\end{equation*}
%\begin{equation*}\tau_{i,i+j}=\text{inf}\{k;~X_{0}=i;X_{k}=i+j\}.\end{equation*}

For ~ $ 0\leq i\leq d$, let $\tau_{i,d}$ be the absorbing time of state $d$ starting from
$i$, i.e.,
$$
\tau_{i,d}:=\inf\{n\geq 1, X_n=d \ \mbox{when} \ X_0=i \},
$$
and $f_{i}(s)$ be the generation function of $\tau_{i,d}$,
\begin{equation}\label{f}
f_{i}(s)=\mathbb{E}s^{\tau_{i,d}}~~~\text{for}~ 0\leq i\leq d,
\end{equation}
we have
\begin{theorem} \label{th1}  For $1\leq j\leq d+1$,
  denote $A_{j}(s)$ as the $d\times d$ sub-matrix by deleting the $(d+1)^{th}$ row and
the $j^{th}$ column of the matrix  $I_{d+1}-sP$. Then, for ~ $ 0\leq i\leq d$, we have
\begin{equation}\label{fi}
f_{i}(s)=(-1)^{d-i}\frac{\det A_{i+1}(s)}{\det A_{d+1}(s)}.
\end{equation}\qed
\end{theorem}

\begin{remark}\label{R1}As a consequence (see Corollary \ref{c1} below), for  the birth-death  and more general  the skip-free (upward jumps {  is} only of unit size, and there is no restriction on downward jumps)
Markov chain  with finite state $\{0, 1, \cdots, d\}$ and absorbing
at state $d$, the absorbing
time  is distributed as a summation  of $d$ independent geometric
(or exponential) random variables.

There are many authors give out different proofs to the results. For the birth and death chain,
the well-known results can be traced back to  Karlin and McGregor (\cite{KM}, 1959) Keilson (\cite{Keillog}, 1971; \cite{Keil}).
Kent and Longford (\cite{Kent}, 1983) proved the result for the discrete time version (nearest random walk) although
they have not specified the result as usual form (section 2, \cite{Kent}). Fill (\cite{F1}, 2009) gave the first stochastic
 proof to both nearest random walk and birth and death chain cases via duality which was established in \cite{DF}.
  Diaconis and Miclo (\cite{DM}, 2009) presented another probabilistic proof for birth and death chain. Gong et al (\cite{Mao}, 2012) gave a similar result in the case that the state space is
  $\mathbb{Z^{+}}$.
For the  skip-free chain, Brown and Shao (\cite{BS}, 1987) first proved  the
result in continuous time situation; Fill
(\cite{F2}, 2009) gave a stochastic proof to both discrete
 and continuous time cases also by using the duality, and  considered the general finite-state Markov chain situation  when the chain starts from {  state $0$}.

However,  the existing proofs we mentioned above are heavily relied on the initial state being $0$, no matter the ``analysis" method by Brown and Shao (\cite{BS}, 1987) and the ``stochastic" method by Fill
(\cite{F2}, 2009), etc..
The first purpose of this paper (Theorem \ref{th1} and \ref{thc}) is to improve the result to the general situation: the chain starts from any state $i$ (not just from state $0$ only (\cite{F2}, 2009)).
 In
particulary, the results generalize the well-known theorems for the
birth-death (Karlin and McGregor \cite{KM}, 1959) and  the skip-free
(\cite{BS} and \cite{F2}) Markov chain.

\end{remark}

Before proving the Theorem, let us at first to recover the results for the skip-free (and then the birth-death) discrete time
Markov chain (Fill \cite{F2}, 2009).

\begin{corollary}\label{c1}
Assume  $p_{i,j}=0$ for $j-i>1$. We have
\begin{equation}\label{f0}
f_{0}(s)=\prod_{i=0}^{d-1} \left[\frac{(1-\lambda_i)s}{1-\lambda_i
s} \right],
\end{equation}
where $\lambda_0, \cdots, \lambda_{d-1}$ are the $d$ non-unit eigenvalues of $P$.

In particular, if all of the eigenvalues are real and nonnegative,
then the hitting time is distributed as the sum of $d$ independent geometric random variables
with parameters $1-\lambda_i$.\\
\end{corollary}
\begin{proof} Note that, $1$ is an eigenvalue of
$P$ evidently. So on the one hand
~${\det(I_{d+1}-sP)}=(1-s)\prod_{i=0}^{d-1}(1-\lambda_i s)$
(where $\lambda_0, \cdots, \lambda_{d-1}$ are the $d$ non-unit
eigenvalues of $P$); on the other hand  we have
${\det (I_{d+1}-sP)}=(1-s)\times \det A_{d+1}(s) $ from (\ref{p}); it's trivial to show that
\begin{equation}\label{A}
 {\det A_{d+1}(s)}=\prod_{i=0}^{d-1}(1-\lambda_i s).
\end{equation}
From the definition of $A_1$, it is easy to see
\begin{equation}\label{A1}
\det A_{1}(s)=(-1)^{d}p_{0,1}p_{1,2}\cdots p_{d-1,d}s^{d}.
\end{equation}
By (\ref{fi}) and (\ref{A}) we have
 \begin{equation*}
 \det A_{1}(1)=(-1)^{d} f_0(1)\cdot {\det A_{d+1}(1)}=(-1)^{d} f_0(1)\cdot
 \prod_{i=0}^{d-1}(1-\lambda_i).\nonumber
 \end{equation*}
And from (\ref{A1}), we get
$\det A_{1}(1)=(-1)^{d} p_{0,1}p_{1,2}\cdots p_{d-1,d}$. Recall that
$f_0(1)=1$,  by (\ref{f}), we obtain
\begin{equation}\label{cp}
 p_{0,1}p_{1,2}\cdots p_{d-1,d}=\prod_{i=0}^{d-1}(1-\lambda_i ).
\end{equation}
Then by (\ref{A1}) and (\ref{cp})
\begin{equation}\label{A2}
\text{det}A_{1}(s)=(-1)^{d} \prod_{i=0}^{d-1}(1-\lambda_i )s^{d},
\end{equation}
and (\ref{f0}) holds from (\ref{A}) and (\ref{A2}) directly .
\end{proof}

\begin{remark}\label{R2}The following example shows that  Theorem \ref{th1} is more convenient compared
with the existing methods in Corollary \ref{c1}, especially we can calculate the expectation of the absorbing time directly from (\ref{fi}).

\end{remark}

Consider a Markov chain with $d+1$ states $\{0, 1, 2, \ldots d\}$
whose transition matrix  $P$ can be  given by
%\begin{equation*}
%    \left \{ \begin{array}{ll}
%p(i,0)=q ,& \mbox{$i=0,1,\cdots,d-1$}; \\
%p(i,i+1)=p , & \mbox{$i=0,1,\cdots,d-1$}; \\
%p(d,d)=1, &  \\
%p(i,j)=0, & \mbox{otherwise.}
%\end{array} \right. \end{equation*} where $p+q=1$.

%The transition matrix $P$ can be written as,
\begin{equation*}\label{p}
P=\left(
 \begin{array}{cccccc}
  q & p  &  & \cdots & 0 & 0 \\
   q & 0  & & \ddots & 0& 0\\
   %q_{2,0} & q_{2,1} & r_3 & p_2 & \\
   %q_{30} & q_{31} & q_{32} & r_4 & p_3 & \\
   \vdots & \ddots  & & \ddots & \ddots & \vdots\\
   q& 0& & \cdots &  0 & p \\
    0 & 0 & & \cdots & 0 & 1\\
     \end{array}
      \right)_{(d+1)\times(d+1)},
\end{equation*} where $p+q=1$.

\begin{corollary}\label{c2} For ~ $ 0\leq i\leq d$,
\begin{equation}\label{aa}
f_{i}(s)=\frac{p^{d-i}s^{d-i}(1-s)+p^{d}qs^{d+1}}{1-s+p^{d}qs^{d+1}},
\end{equation}
and we have
\begin{equation}\label{bb}
\mathbb{E}{\tau_{i,d}}=\frac{1-p^{d-i}}{p^{d}q}.
\end{equation}
\end{corollary}
\noindent{\it Proof }.
Take full advantage of $p+q=1$, we can calculate that
\begin{equation*}
\det A_{i+1}(s)=(-1)^{d-i}\frac{p^{d-i}s^{d-i}(1-s)+p^{d}qs^{d+1}}{1-ps},
\end{equation*}
\begin{equation*}
\det A_{d+1}(s)=\frac{1-s+p^{d}qs^{d+1}}{1-ps}.
\end{equation*}
We obtain (\ref{aa}) by using Theorem $2.1$. Recall that $\mathbb{E}{\tau_{i,d}}=f_{i}^{'}(1)$, we can get (\ref{bb}) by some calculation easily. \qed

 \noindent{\it Proof of Theorem \ref{th1}} By decomposing the first step, for $0\leq i\leq d$, the generation function of $\tau_{i,d}$ satisfies,
 \begin{equation*}
 % \label{key}
 \begin{split} f_{i}(s)&=r_{i}sf_{i}(s)+p_{i,i+1}sf_{i+1}(s)+p_{i,i+2}sf_{i+2}(s)+\cdots p_{i,d-1}sf_{d-1}(s)+p_{i,d}s\\
  &~~~~~q_{i,i-1}sf_{i-1}(s)+q_{i,i-2}sf_{i-2}(s)+\cdots+q_{i,0}sf_{0}(s).
\end{split}
\end{equation*}
These  system of equations are linear with respect to $f_{0}(s),~f_{1}(s)\cdots,~f_{d-1}(s)$. Using Cramer's Rule,
we can get (\ref{fi}) by solving from these equations. \qed

\subsection{ Passage time when starting from the stationary distribution}

{  Consider a discrete time Markov chain $\{X_n\}_{n\geq 1}$  with finite states $\{0, 1, \cdots, d\}$ starting from the stationary distribution $\pi:=(\pi_0,\pi_1,\dots,\pi_d)$,} the transition probability matrix $\widehat{P}$ is given by
\begin{equation*}\label{p}
\widehat{P}=\left(
 \begin{array}{cccccc}
  r_0 & p_{0,1}  &  & \cdots & p_{0,d-1} & p_{0,d} \\
   q_{1,0} & r_1  & & \cdots & p_{1,d-1}& p_{1,d}\\
   %q_{2,0} & q_{2,1} & r_3 & p_2 & \\
   %q_{30} & q_{31} & q_{32} & r_4 & p_3 & \\
   \vdots & \ddots & & \ddots & \ddots & \vdots\\
    q_{d,0} & q_{d,1} & & \cdots & q_{d,d-1} & r_d\\
     \end{array}
      \right)_{(d+1)\times(d+1)}.
\end{equation*}
In addition, write
\begin{equation}\label{d}
 D=\left(
 \begin{array}{cccccc}
  1 & -\pi_{1}  &  & \cdots & -\pi_{d-1} & -\pi_{d} \\
  0 & 1-r_{1}s  & & \cdots & -p_{1,d-1}s& -p_{1,d}s\\
   %q_{2,0} & q_{2,1} & r_3 & p_2 & \\
   %q_{30} & q_{31} & q_{32} & r_4 & p_3 & \\
   \vdots & \ddots  & & \ddots & \ddots & \vdots\\
    0 & -q_{d,1}s & & \cdots & -q_{d,d-1}s & 1-r_ds\\
     \end{array}
      \right),
 \end{equation} and
 \begin{equation}\label{d0}
 D_{0}=\left(
 \begin{array}{cccccc}
  \pi_{0} & -\pi_{1}  &  & \cdots & -\pi_{d-1} & -\pi_{d} \\
  q_{1,0}s  & 1-r_{1}s & & \cdots & -p_{1,d-1}s& -p_{1,d}s\\
   %q_{2,0} & q_{2,1} & r_3 & p_2 & \\
   %q_{30} & q_{31} & q_{32} & r_4 & p_3 & \\
   \vdots & \ddots & & \ddots & \ddots & \vdots\\
    q_{d,0}s  & -q_{d,1}s & & \cdots & -q_{d,d-1}s & 1-r_ds\\
     \end{array}
      \right).
 \end{equation}
 \begin{theorem}\label{gsta}
 When the chain starts  from $X_0$ with the stationary distribution $\pi$, and
 \begin{equation}\label{tao}
\tau:=\inf\{n\geq 0, X_n=0 \ \},
\end{equation}
 be the passage  time of state $0$.
Denote  $g_{\pi}(s)$ as the generation function of $\tau$, i.e., $g_{\pi}(s)=\mathbb{E}_{\pi}s^{\tau}$;
 we have
\label{ths} \begin{equation*}
g_{\pi}(s)=\frac{\det D_{0}}{\det D},\end{equation*}where $D$ and $D_{0}$ are given in (\ref{d}) and (\ref{d0}) respectively.
\end{theorem}

{  Specifically, if the chain is time reversible and ergodic, Brown(\cite{B}, 1999) point out the elegant connection between the passage time stating from the stationary distribution and the interlacing eigenvalues theorem of linear algebra.
Recently, this result also proved by  Fill and Lyzinski(\cite{FL}, 2013) with  a   stochastic  method.
%We reprove (directly and simply) the results by  Brown(\cite{B}, 1999), which also proved recently  by  Fill and Lyzinski(\cite{FL}, 2013) with  a   stochastic  method.
 In what follows, we will reprove it directly as a corollary of Theorem \ref{gsta}.}

 \begin{corollary}\label{csta} \label{Ppi}Let $\lambda_1, \cdots, \lambda_{d}$ be the $d$ non-unit eigenvalues of $\widehat{P}$ (we assume $\lambda_0=1$), and  $\gamma_1, \cdots, \gamma_{d}$ be the $d$ eigenvalues of $\widehat{P}_{0}$, which is the sub-matrix obtained by deleting the first row and the first column of $\widehat{P}$.
Then we have
\begin{equation*}
%\label{gpi}
g_{\pi}(s)=\left(\prod_{i=1}^{d}\frac{1-\gamma_{i}}{1-\lambda_{i}}\right) \prod_{i=1}^{d}\left(\frac{1-\lambda_is}{1-\gamma_i
s}\right).
\end{equation*}
\end{corollary}

\noindent {\it Proof of Corollary \ref{csta}} ~~
   It is easy to see that (recall $\gamma_1, \cdots, \gamma_{d}$ are the eigenvalues of $\widehat{P}_{0}$),
\begin{equation} \label{D}
\det D=\prod_{i=1}^{d}(1-\gamma_i s).
 \end{equation}

In what follows we will show
\begin{equation}\label{D0}
\det D_{0}=\pi_{0}\prod_{i=1}^{d}(1-\lambda_is).
 \end{equation}
 Define $e_1=(1,0,\dots,0)$, and recall that $\pi=(\pi_0,\pi_1,\dots,\pi_d)$. Then, we have \begin{equation*}
\det D_{0}=\left |
 \begin{array}{cc}
  0 &  \pi\\
  e_1^{T} & I-s\widehat{P}
  \end{array}
  \right |.
 \end{equation*} If we let $\Pi=\text{diag}(\pi_0,\pi_1,\dots,\pi_d)$, the reversibility of $\widehat{P}$ implies that $\Pi^{1/2}\widehat{P}\Pi^{-1/2}$ is a real symmetric matrix. Thus there exist an orthogonal matrix $U$ such that
 \begin{equation}\label{lf1}
U\Pi^{1/2}\widehat{P}\Pi^{-1/2}U^{T}=\text{diag}(\lambda_{0},\lambda_{1},\dots,\lambda_{d}).
 \end{equation} We can calculate that
 \begin{equation*} \begin{split}
& \left (
 \begin{array}{cc}
  1 &  0\\
  0 & U
  \end{array}
  \right )
 \left (
 \begin{array}{cc}
  1 &  0\\
  0 & \Pi^{1/2}
  \end{array}
  \right )
  \left (
 \begin{array}{cc}
  0 &  \pi\\
  e_1^{T} & I-s\widehat{P}
  \end{array}
  \right )
   \left (
  \begin{array}{cc}
  1 &  0\\
  0 & \Pi^{-1/2}
  \end{array}
  \right )
   \left (
  \begin{array}{cc}
  1 &  0\\
  0 & U^{T}
  \end{array}
  \right )\\
 =&
  \left (
  \begin{array}{cc}
  0 &  \pi\Pi^{-1/2}U^{T}\\
  U\Pi^{1/2}e_{1}^{T} & U\Pi^{1/2}(I-s\widehat{P})\Pi^{-1/2}U^{T}
  \end{array}
  \right )=\left (
  \begin{array}{cc}
  0 &  \pi\Pi^{-1/2}U^{T}\\
  U\Pi^{1/2}e_{1}^{T} &I-s\text{diag}(\lambda_{0},\lambda_{1},\cdots,\lambda_{d})
  \end{array}
  \right )
\end{split} \end{equation*}

{ It is easy to prove that $\lambda_{0}=1$ is the unique eigenvalue of maximum modulus of $\widehat{P}$. So the geometric multiplicity of $\widehat{P}$ corresponding to $\lambda_{0}$ is one (\cite{KM} P500 Perron's Theorem). On the one hand, $e_{1}U\Pi^{1/2}$ is {  a left eigenvector} corresponding to $\lambda_{0}$. $\pi$ is also the left eigenvector of $\lambda_{0}$. Because $\|e_{1}U\Pi^{1/2}\|=\|\pi\|=1$, we have
$e_{1}U\Pi^{1/2}=\pi$. So {  \begin{equation}\label{lff}
U\Pi^{1/2}e_{1}^{T}=(e_{1}U\Pi^{1/2})^{T}=\pi^{T}.\end{equation}}
On the other hand, $\Pi^{-1/2}U^{T}e_{1}^{T}$ is {  a right eigenvector} corresponding to eigenvalue $\lambda_{0}$, and $\mathbf{1}=\{1,1,\cdots,1\}$ is also the right eigenvector of $\lambda_{0}$. Because $\|\Pi^{-1/2}U^{T}e_{1}^{T}\|=\|\mathbf{1}\|=1$, we have $\Pi^{-1/2}U^{T}e_{1}^{T}=\mathbf{1}$, and \begin{equation}\label{lf2}\pi\Pi^{-1/2}U^{T}e_{1}^{T}=\pi\mathbf{1}=1\end{equation}
By (\ref{lf1}) and $\pi=\pi\widehat{P}$, \begin{equation*}\pi\Pi^{-1/2}U^{T}=\pi\widehat{P}\Pi^{-1/2}U^{T}=\pi\Pi^{-1/2}U^{T}\text{diag}(\lambda_{0},\lambda_{1},\dots,\lambda_{d}).\end{equation*}
Because $\lambda_{i}\neq 1$ for $i=1,2,\cdots d$, by (\ref{lf2}), $\pi\Pi^{-1/2}U^{T}$ must be equal to $e_1$.
By (\ref{lff}), we obtain
\begin{equation*}
\det D_{0}=\left |
  \begin{array}{cc}
  0 &  \pi\Pi^{-1/2}U^{T}\\
  U\Pi^{1/2}e_{1}^{T} &I-s\widehat{P}
  \end{array}
  \right |=\left |
  \begin{array}{cc}
  0 &  e_1\\
  \pi^{T} &I-s\widehat{P}
  \end{array}
  \right |=\pi_{0}\prod_{i=1}^{d}(1-\lambda_is),
\end{equation*}}
which we get (\ref{D0}).
Combine (\ref{D}) and (\ref{D0}), we obtain
 \begin{equation*}
g_{\pi}(s)=\frac{\pi_{0}\prod_{i=1}^{d}(1-\lambda_is)}{\prod_{i=1}^{d}(1-\gamma_i s)}.\end{equation*} Because $g_{\pi}(s)$ is a generation function, $g_{\pi}(1) = 1$. So
 \begin{equation*}
\pi_{0}=\frac{\prod_{i=1}^{d}(1-\gamma_i)}{\prod_{i=1}^{d}(1-\lambda_i)},\end{equation*} which complete the proof.
\qed

\noindent {\it Proof of Theorem \ref{gsta}} ~~ Denote $g_{i}(s)$ as the generation function of passage time of state $0$ when the chain is starting from $i$. By the Markov property, we have
\begin{equation}\label{keypi1}
g_{\pi}(s)=\pi(0)g_{0}(s)+\pi(1)g_{1}(s)+\cdots+\pi(d)g_{d}(s).
\end{equation}
Obviously, $g_{0}(s)=1$. By decomposing the first step, for $1\leq i\leq d$, $g_{i}(s)$ satisfies
 \begin{equation*}
\begin{split}
g_{i}(s)&=r_{i}sg_{i}(s)+p_{i,i+1}sg_{i+1}(s)+\cdots p_{i,d-1}sg_{d-1}(s)+p_{i,d}sg_{d}(s)\\
  &~~~~~+q_{i,i-1}sg_{i-1}(s)+q_{i,i-2}sg_{i-2}(s)+\cdots+q_{i,0}s.
\end{split} \end{equation*}
These system of equations together with (\ref{keypi1}) are linear with respect to $g_{\pi}(s),~g_{1}(s),~g_{2}(s)\cdots,~g_{d}(s)$. Use Cramer's Rule, we can get $g_{\pi}(s)$  by solving from these equations as
 \begin{equation*}
g_{\pi}(s)=\frac{\det D_{0}}{\det D}.\end{equation*}
\qed
\begin{remark}\label{R4}Actually, if we define for $i=1,2,\cdots, d$
\begin{equation}\label{tao}
\tau_{i}:=\inf\{n\geq 0, X_n=i \ \},
\end{equation}
 be the passage  time of state $i$.
Denote  $g_{\pi}^{i}(s)$ as the generation function of $\tau_{i}$, i.e., $g_{\pi}^{i}(s)=\mathbb{E}_{\pi}s^{\tau_{i}}$, we can obtain the formula for $g_{\pi}^{i}(s)$ with the corresponding modification for the $D$ and $D_{0}$, the proof is almost line by line with regard of $g_{i}(s)=1$ this time .
\end{remark}

%%%%%%%%%%%%%%%%%%%%%%%%%%%%% Section 3%%%%%%%%%%%%%%%%%%%%%%%%%%%%%%%%%%%%\\
\section{ Continuous time\label{s3}}

We can write the counterpart results for the continuous time Markov chain with finite states $\{0, 1, \cdots, d\}$ easily. The proof is similar as in section 1 and so we omit the details.

\subsection{Starting from any fixed state }
Define $\{X_t\}_{t\geq 0}$ being the (continuous time) Markov chain with finite states $\{0, 1, \cdots, d\}$ and absorbing
at state $d$,  the generator $Q$ is given by
\begin{equation*}
Q=\left(
 \begin{array}{cccccc}
  -\gamma_0 & \alpha_{0,1}  &  & \cdots & \alpha_{0,d-1} & \alpha_{0,d} \\
   \beta_{1,0} & -\gamma_1 & & \cdots & \alpha_{1,d-1}& \alpha_{1,d}\\
   %q_{2,0} & q_{2,1} & r_3 & p_2 & \\
   %q_{30} & q_{31} & q_{32} & r_4 & p_3 & \\
   \vdots & \ddots & & \ddots & \ddots & \vdots\\
   \beta_{d-1,0}& \beta_{d-1,1}& & \cdots &  -\gamma_{d-1} & \alpha_{d-1,d} \\
    0 & 0 & & \cdots & 0 & 0\\
     \end{array}
      \right)_{(d+1)\times(d+1)}.
\end{equation*}
 Let $\tau_{i,d}$ be the
absorbing time of state $d$ starting from $i$ and
$\widetilde{f}_{i}(s)$ be
 the Laplace transform of $\tau_{i,d}$. i.e.
\begin{equation*}\widetilde{f}_{i}(s)=\mathbb{E}e^{-s\tau_{i,d}}. \end{equation*}

 \begin{theorem}\label{thc}
For $1\leq j\leq d+1$, we denote $\widetilde{A}_{j}(s)$ as the $d\times d$ sub-matrix by deleting the $(d+1)^{th}$ row and
the $j^{th}$ column of the matrix  $sI_{d+1}-Q$. Then, for $0\leq i\leq d$
we have
\begin{equation}\label{fi1}
\widetilde{f}_{i}(s)=(-1)^{d-i}\frac{\det\widetilde{A}_{i+1}}{\det\widetilde{A}_{d+1}}.
\end{equation} \qed
\end{theorem}

Immediately, we recover the results for the skip-free continuous time
Markov chain (Brown and Shao \cite{BS}, 1987).

 \begin{corollary} Assume  $\alpha_{i,j}=0$ for $j-i>1$. We have
\begin{equation*}
\varphi_{d}(s)=\prod_{i=0}^{d-1}\frac{\lambda_i}{\lambda_i+s},
\end{equation*}
where $\lambda_i$ are the $d$ non-zero eigenvalues of $-Q$.

 In particular, if all of the eigenvalues are real and nonnegative, then the hitting time is distributed as
 the sum of $d$ independent exponential random variables with parameters $\lambda_i$.
 \end{corollary}
 \begin{proof} The proof is similar as Corollary \ref{c1}, we can
 calculate that
${\det\widetilde{A}_{d+1}}=\prod_{i=0}^{d-1}(\lambda_i+s)$, and  $\det\widetilde{A}_{1}=(-1)^{d}\alpha_{0,1} \alpha_{1,2}\cdots
\alpha_{d-1,d}=(-1)^{d}\prod_{i=0}^{d-1}\lambda_i$.
\end{proof}

\subsection{Starting from the stationary distribution}
If we consider a time reversible ergodic Markov chain with generator $\widehat{Q}$, let $\widehat{Q}_{0}$ be the sub-matrix which is obtained by deleting the first row and the first column of $\widehat{Q}$.
We denote $\widetilde{g}_{\pi}(s)$ as the Laplace transform of the hitting time of state $0$ when the chain is starting from the stationary distribution $\pi$.

\begin{theorem} We have
\begin{equation}\label{gpi2}
\widetilde{g}_{\pi}(s)=\left(\prod_{i=1}^{d}\frac{\gamma_{i}}{\lambda_{i}}\right) \prod_{i=1}^{d}\left(\frac{\lambda_i+s}{
\gamma_i+s}\right),
\end{equation}
where $\lambda_1, \cdots, \lambda_{d}$ are the $d$ non-zero eigenvalues of $-\widehat{Q}$ (we assume $\lambda_0=0$), and  $\gamma_1, \cdots, \gamma_{d}$ are the $d$ eigenvalues of $-\widehat{Q}_{0}$.
 \end{theorem}

\bigskip
\noindent\textbf{Acknowledgement}

{  We  appreciate Professor Daniel R. Jeske for his interesting example (see Corollary \ref{c2}) . We would also thank the referee's valuable suggestions.}

\bigskip

%%%%%%%%%%%%%%%%%%%%%%%%%%%%% References %%%%%%%%%%%%%%%%%%%%%%%%%%%%%%%%%%%%\\

\end{document}